\documentclass[draft]{amsart}

\usepackage{pgf,tikz}
\usetikzlibrary{arrows}
\usepackage{cite}

\usepackage{amssymb,graphicx,enumerate,color}
\usepackage{multirow}

\newcommand{\C}{\tilde{C}}

\renewcommand{\C}{\mathcal{C}}
\renewcommand{\P}{\mathcal{P}}
\newcommand{\U}{\mathcal{U}}

\newcommand{\R}{{\mathcal{R}}}
\newcommand{\Q}{\mathcal{Q}}

\newcommand{\beqs}{\begin{equation*}}
\newcommand{\eeqs}{\end{equation*}}
\numberwithin{equation}{section}
 \theoremstyle{plain}
\newtheorem{theorem}{Theorem}[section]
\newtheorem{lemma}[theorem]{Lemma}

\theoremstyle{remark}

\begin{document}

\makeatletter
\def\imod#1{\allowbreak\mkern10mu({\operator@font mod}\,\,#1)}
\makeatother

\author{Alexander Berkovich}
   \address{Department of Mathematics, University of Florida, 358 Little Hall, Gainesville FL 32611, USA}
   \email{alexb@ufl.edu}

\author{Ali Kemal Uncu}
   \address{Department of Mathematics, University of Florida, 358 Little Hall, Gainesville FL 32611, USA}
   \email{akuncu@ufl.edu}

\title[New Weighted Partition Theorems]{New Weighted Partition Theorems with the Emphasis on the Smallest Part of Partitions}
     
\begin{abstract} We use the $q$-binomial theorem, the $q$-Gauss sum, and the ${}_2\phi_1 \rightarrow {}_2\phi_2$ transformation of Jackson to discover and prove many new weighted partition identities. These identities involve unrestricted partitions, overpartitions, and partitions with distinct even parts. Smallest part of the partitions plays an important role in our analysis. This work was motivated in part by the research of Krishna Alladi. 
\end{abstract}   
   


\date{\today}
   
\maketitle
\section{Introduction}

A \textit{partition}, $\pi=(\lambda_1,\lambda_2,\dots)$, is a finite sequence of non-increasing positive integers. The empty sequence is conventionally considered to be the unique partition of zero. The elements $\lambda_i$ that appear in the list $\pi$ are called \textit{parts} of the partition $\pi$. The sum of all parts of a partition is called the \textit{norm} of a partition $\pi$. We call a partition $\pi$ \textit{a partition of} $n$ if its norm is $n$.

We list some useful statistics/notations that will be used in the paper. Given a partition $\pi$,
\begin{align*}
|\pi| \ &:= \text{norm of the partition } \pi,\\
s(\pi) \ &:=\text{smallest part of the partition }\pi,\\
\nu_e(\pi)\ &:= \text{number of even parts in }\pi,\\
\nu_o(\pi)\ &:= \text{number of odd parts in }\pi,\\
\nu(\pi) &:= \nu_e(\pi) + \nu_o(\pi) = \text{number of parts in }\pi,\\
\nu_{d}(\pi)\ &:= \text{number of different parts in }\pi.
\end{align*}
For example, $\pi = (10,9,5,5,4,1,1)$ is a partition of 35 with $s(\pi)=1$, $\nu(\pi)=7$, and $\nu_d(\pi)=5$. 

Alladi studied the weighted partition identities methodically. In 1997, among many interesting results, he noted the general identity:

\begin{theorem} [Alladi, 1997] \label{Overpartitions_general_theorem} Let $a$, $b$ and $q$ be variables.
\begin{equation}\label{overpartitions_general_GF}
\frac{(a(1-b)q;q)_n}{(aq;q)_n} = 1+ \sum_{\pi\in \mathcal{U}_{ n}} a^{\nu(\pi)} b^{\nu_d(\pi)} q^{|\pi|},
\end{equation} where $\U_{n}$ is the set of non-empty ordinary partitions into parts $\leq n$.
\end{theorem}
In \eqref{overpartitions_general_GF} and in the rest of the paper we use the standard q-Pochhammer symbol notations defined in \cite{Theory_of_Partitions}, \cite{GasperRahman}. Let $L$ be a non-negative integer, then
\begin{equation*}
(a;q)_L := \prod_{i=0}^{L-1} (1-aq^{i}) \text{  and  } (a;q)_\infty :=\lim_{L\rightarrow\infty} (a;q)_L.
\end{equation*}
 
 We now discuss a special case of the Theorem~\ref{Overpartitions_general_theorem} that  plays important role in the study of overpartitions. We define an overpartition to be a partition where the last appearance of a part may come with a mark (usually put as an overhead bar on the part, hence the name). Any partition is an overpartition of the same number. One non-trivial example is $\bar{\pi}=(10,\bar{9},5,5,4,1,\bar{\tiny 1})$. All the statistics defined above translate in the obvious manner to overpartitions. The definition, the interpretation of overpartitions and the generating function for the number of overpartitions are given by Corteel and Lovejoy in their influential paper \cite{overpartitions}. 

We would like to define the following sets:
\begin{align*}
\U &:= \text{the set of all non-empty ordinary partitions,}\\
\mathcal{O} &:= \text{the set of all non-empty overpartitions.}
\end{align*}

The connection of the identity \eqref{overpartitions_general_GF} and overpartitions can be seen by setting $a=1$ and $b=2$, \begin{equation}\label{weights_of_overpartitions}
\frac{(-q;q)_n}{(q;q)_n} =1+ \sum_{\pi \in \mathcal{U}_{ n}} 2^{\nu_d(\pi)}q^{|\pi|}.
\end{equation}
The left side of the identity \eqref{weights_of_overpartitions} is interpreted as the generating function for the number of overpartitions into parts $\leq n$. The right-hand side of \eqref{weights_of_overpartitions} is the weighted connection of overpartitions with ordinary partitions. We can write the weighted connection between ordinary partitions and overpartitions abstractly \begin{equation}\label{over_connect_regular_abstract}
\sum_{\pi \in \mathcal{O}_{ n}} q^{|\pi|} = \sum_{\pi \in \mathcal{U}_{ n}} 2^{\nu_d(\pi)}q^{|\pi|}, \end{equation}where $\mathcal{O}_{ n}$ is the set of non-empty overpartitions into parts $\leq n$.

The equation \eqref{over_connect_regular_abstract} is an example of a weighted partition identity between
sets of partitions. In this paper we prove new weighted partition identities involving statistics other than $2^{\nu_d(\pi)}$. 

Section~\ref{Section_def} has necessary definitions and identities to follow the results in the paper. The weighted partition identities for ordinary partitions and overpartitions for the smallest part, $s(\pi)$, is given in Section~\ref{Section_s}. A weighted count of overpartitions' relation with the number of representations of a number as a sum of two squares will be given in Section~\ref{Section_s_v}. Section~\ref{Section_P_ped} has weighted partition identities related to partitions with distinct even parts. In Section~\ref{Section_not_0_mod_3} we provide a more involved weighted identity involving overpartitions into parts not divisible by 3.

\section{Definition and Background Information}\label{Section_def}

Partitions can be represented in the \textit{frequency notation} $\pi = (1^{f_1},2^{f_2},\dots)$ by writing parts of $\pi$ in a finite sequence format with exponents, where the exponents $f_i(\pi)$ of the natural numbers denote the number of appearances of that part in $\pi$. We abuse the notation and write $f_i$, frequency of $i$, when the partition is understood from the context. Similarly, we drop the zero frequencies in our notation to keep the notations neat. A zero frequency may still be used to indicate and stress an integer not being a part of a partition. For example the partition $\pi = (10,9,5,5,4,1,1)$ can be represented in the frequency notation as  $(1^2,2^0,3^0,4^1,5^2,6^0,7^0,8^0,9^1,10^1,11^0,\dots) = (1^2,4,5^2,7^0,9,10)$. Here $\pi$ is a partition of 35 where the frequency of 1, $f_1(\pi)=f_1=2$, $f_4=1$, $f_5=2\dots$ and the integer $7$ is not a part of $\pi$.

One can also extend the frequency notation to overpartitions by allowing sequence elements with a positive frequency to have an overhead bar meaning that the first appearance of that part is marked. The norm of overpartitions and ordinary partitions are defined the same way. In the frequency notation we can represent $\bar{\pi}$ as  $(\bar{\tiny 1}^2,4,5^2,\bar{9},10)$. 

Other representations of partitions include the \textit{Ferrers diagrams} and \textit{2-modular Ferrers diagrams} \cite[$\S$1.3]{Theory_of_Partitions}. The Ferrers diagram and the 2-modular Ferrers diagram are formed by drawing rows boxes where the row sum (count of boxes or the sum of the contents, respectively) adds up to the corresponding part of the partition. It should be reminded to the reader that in the 2-modular diagrams, only the boxes at the end of a row may be filled by 1 or 2; all the other boxes are filled with 2's. An example of the Ferrers diagram and a 2-modular Ferrers diagram is given in Table~\ref{Table_Ferrers_Diagrams}.

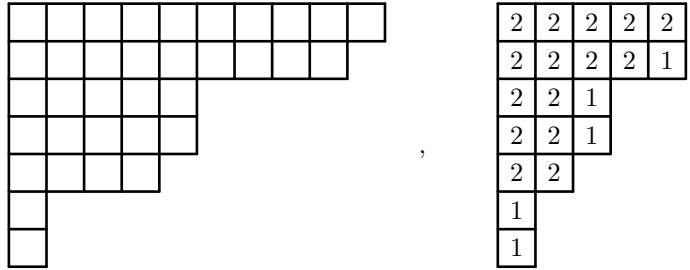
\begin{center}
\begin{table}[htb]\caption{The Ferrers Diagram and the 2-modular Ferrers Diagram of the partition $\pi=(10,9,5,5,4,1,1)$.}\label{Table_Ferrers_Diagrams}
\definecolor{cqcqcq}{rgb}{0.75,0.75,0.75}
\begin{tikzpicture}[line cap=round,line join=round,>=triangle 45,x=0.5cm,y=0.5cm]
\clip(0.5,-0.1) rectangle (19.5,7.5);
\draw [line width=1pt] (1,0)-- (2,0);
\draw [line width=1pt] (1,1)-- (2,1);
\draw [line width=1pt] (2,0)-- (2,7);
\draw [line width=1pt] (1,7)-- (11,7);
\draw [line width=1pt] (11,7)-- (11,6);
\draw [line width=1pt] (11,6)-- (1,6);
\draw [line width=1pt] (1,7)-- (1,0);
\draw [line width=1pt] (1,2)-- (5,2);
\draw [line width=1pt] (5,2)-- (5,7);
\draw [line width=1pt] (10,7)-- (10,5);
\draw [line width=1pt] (10,5)-- (1,5);
\draw [line width=1pt] (9,5)-- (9,7);
\draw [line width=1pt] (8,7)-- (8,5);
\draw [line width=1pt] (7,7)-- (7,5);
\draw [line width=1pt] (6,7)-- (6,3);
\draw [line width=1pt] (6,3)-- (1,3);
\draw [line width=1pt] (1,4)-- (6,4);
\draw [line width=1pt] (4,2)-- (4,7);
\draw [line width=1pt] (3,7)-- (3,2);
\draw [line width=1pt] (14,7)-- (19,7);
\draw [line width=1pt] (19,7)-- (19,5);
\draw [line width=1pt] (19,5)-- (14,5);
\draw [line width=1pt] (14,7)-- (14,0);
\draw [line width=1pt] (15,0)-- (15,7);
\draw [line width=1pt] (17,3)-- (17,7);
\draw [line width=1pt] (18,5)-- (18,7);
\draw [line width=1pt] (16,7)-- (16,2);
\draw [line width=1pt] (19,6)-- (14,6);
\draw [line width=1pt] (17,4)-- (14,4);
\draw [line width=1pt] (17,3)-- (14,3);
\draw [line width=1pt] (16,2)-- (14,2);
\draw [line width=1pt] (15,1)-- (14,1);
\draw [line width=1pt] (15,0)-- (14,0);
\draw (14.5,6.5) node[anchor=center] { 2};
\draw (15.5,6.5) node[anchor=center] { 2};
\draw (16.5,6.5) node[anchor=center] { 2};
\draw (17.5,6.5) node[anchor=center] { 2};
\draw (18.5,6.5) node[anchor=center] { 2};
\draw (17.5,5.5) node[anchor=center] { 2};
\draw (16.5,5.5) node[anchor=center] { 2};
\draw (15.5,5.5) node[anchor=center] { 2};
\draw (14.5,5.5) node[anchor=center] { 2};
\draw (14.5,4.5) node[anchor=center] { 2};
\draw (15.5,4.5) node[anchor=center] { 2};
\draw (15.5,3.5) node[anchor=center] { 2};
\draw (15.5,2.5) node[anchor=center] { 2};
\draw (14.5,3.5) node[anchor=center] { 2};
\draw (14.5,2.5) node[anchor=center] { 2};
\draw (18.5,5.5) node[anchor=center] { 1};
\draw (16.5,4.5) node[anchor=center] { 1};
\draw (16.5,3.5) node[anchor=center] { 1};
\draw (14.5,1.5) node[anchor=center] { 1};
\draw (14.5,0.5) node[anchor=center] { 1};
\draw (12,3) node[anchor=center] {,};
\end{tikzpicture}
\end{table}
\end{center}

Note that the \textit{conjugate} of a Ferrers diagram (drawing a Ferrers diagram column-wise and reading it row-wise) is also a partition. The conjugate of $\pi$ is $(7,5,5,5,4,2,2,2,2,1)$. Conjugation doesn't extend to 2-modular graphs directly. A partition's 2-modular diagram yields another 2-modular diagram under conjugation only when the original partition has distinct odd parts. The conjugate of the 2-modular graph of the example in Table~\ref{Table_Ferrers_Diagrams} does not yield a permissible 2-modular diagram.

Ferrers diagrams can be extended to overpartitions. One can easily mark the rows of the Ferrers diagrams by coloring the box at the end of the row to indicate that the related part of the partition is overlined. Conjugation of the ordinary Ferrers diagrams carry over for overpartitions without a hitch. It is easy to check that the conjugate of $\bar{\pi}=(10,\bar{9},5,5,4,1,\bar{\tiny 1})$ is $(\bar{7},5,5,5,4,2,2,2,\bar{\tiny 2},1)$.

We define the basic $q$-hypergeometric series as they appear in \cite{GasperRahman}. Let $r$ and $s$ be non-negative integers and $a_1,a_2,\dots,a_r,b_1,b_2,\dots,b_s,q,$ and $z$ be variables. Then \begin{equation*}_r\phi_s\left(\genfrac{}{}{0pt}{}{a_1,a_2,\dots,a_r}{b_1,b_2,\dots,b_s};q,z\right):=\sum_{n=0}^\infty \frac{(a_1;q)_n(a_2;q)_n\dots (a_r;q)_n}{(q;q)_n(b_1;q)_n\dots(b_s;q)_n}\left[(-1)^nq^{n\choose 2}\right]^{1-r+s}z^n.\end{equation*}
Let $a$, $b$, $c$, $q$, and $z$ be variables. The $q$-binomial theorem \cite[II.4, p. 236]{GasperRahman} is \begin{equation}\label{q_binomial}
{}_1\phi_0 \left(\genfrac{}{}{0pt}{}{a}{-};q,z \right) = \frac{(az;q)_\infty}{(z;q)_\infty},
\end{equation} 
and the $q$-Gauss sum \cite[II.8, p. 236]{GasperRahman} is \begin{equation}\label{q_Gauss}
{}_2\phi_1 \left(\genfrac{}{}{0pt}{}{a,\ b}{c};q,{c}/{ab} \right) = \frac{({c}/{a};q)_\infty ({c}/{b};q)_\infty}{(c;q)_\infty ({c}/{ab};q)_\infty}.
\end{equation} The Jackson ${}_2\phi_1$ to ${}_2\phi_2$ transformation \cite[III.4, p. 241]{GasperRahman} is \begin{equation}\label{Jackson_transformation}
{}_2\phi_1 \left(\genfrac{}{}{0pt}{}{a,b}{c};q,z\right) = \frac{(az;q)_\infty}{(z;q)_\infty}{}_2\phi_2 \left(\genfrac{}{}{0pt}{}{a,c/b}{c,az};q,bz\right).
\end{equation}
We would also like to recall the definition of the classical theta functions $\varphi$ and $\psi$ \begin{equation}\label{Theta_phi}
\varphi(q):=\sum_{n=-\infty}^\infty q^{n^2},\text{ and }\psi(q):=\sum_{n=0}^\infty q^{n(n+1)/2}.
\end{equation} The Gauss identities \cite[Cor 2.10, p. 23]{Theory_of_Partitions} for these functions will be of use:
\begin{align}
\label{Theta_phi_prod}
\varphi(-q) &= \sum_{n=-\infty}^\infty(-1)^n q^{n^2} = \frac{(q;q)_\infty}{(-q;q)_\infty},\\
\label{Theta_psi_product}\psi(q) &= \frac{(q^2;q^2)_\infty}{(q;q^2)_\infty}.
\end{align}


\section{Weighted Identities with respect to the Smallest Part of a Partition}\label{Section_s}

Let $\U^*$ be the subset of $\U$  such that for every $\pi\in \U^*$, $f_1(\pi)\equiv 1\mod{2}$. Next, we introduce a new partition statistic $t(\pi)$ to be the number defined by the properties
\begin{enumerate}[i.]
\item $f_i\equiv1\mod{2}$, for $1\leq i\leq t(\pi)$, 
\item and $f_{t(\pi)+1}\equiv 0 \mod{2}$.
\end{enumerate}  Note that for any $\pi\in \U$ with an even frequency of 1 (which could be 0) we have $t(\pi)=0$. Then we have the weighted partition identity between the set of ordinary partitions and its subset $\U^*$ as follows.

\begin{theorem}\label{Ordinary_Partitions_Combinatorial_Weighted_Theorem}
\begin{equation}\label{Ordinary_Partitions_Combinatorial_Weighted_Equation}
\sum_{\pi\in\U} (-1)^{s(\pi)+1} q^{|\pi|} = \sum_{\pi\in\U^*} t(\pi) q^{|\pi|}.
\end{equation}
\end{theorem}

The left side identity is the weighted count of partitions of a given norm $n$ where every partition with an odd smallest part gets counted with $+1$ and the partitions of $n$ with an even smallest part gets counted with $-1$. There are $42$ partitions of 10 in total. From this number, 9 partitions, $(2^5)$,  $(2^3,4)$, $(2^2,3^2)$, $(2^2,6)$, $(2,3,5)$, $(2,4^2)$, $(2,8)$, $(4,6)$, $(10)$, have an even smallest part. Therefore, from the count of the left-hand side of \eqref{Ordinary_Partitions_Combinatorial_Weighted_Equation}, the coefficient of the $q^{10}$ is $24 = 42 - 2\cdot 9 $. The right-hand side count and the weights can be found in Table~\ref{Table_Ordinary_Partitions_t_weight}.

\begin{table}[htb]\caption{Example of Theorem~\ref{Ordinary_Partitions_Combinatorial_Weighted_Theorem} with $|\pi|=10$.}\label{Table_Ordinary_Partitions_t_weight}
\begin{center}\vspace{-.5cm}
\[\begin{array}{cc|cc}
\pi\in\U^*  &  t(\pi) &  \pi\in\U^*  &  t(\pi) \\[-2ex]& &  \\
(1,2,3,4) 	& 4		&   (1,4,5) 		& 1		\\
(1,2^3,3)	& 3		&   (1,2^2,5) 		&  1		\\
(1^5,2,3)		& 3		&   (1^5,5) 		&	 1	\\
(1,2,7) 	& 2		&   (1^3,3,4)	 	&	1	\\
(1^3,2,5) 	& 2		&   (1,3^3)  		& 1	\\
(1,9)		& 1		&	(1^3,2^2,3)		& 1\\
(1^3,7) 	& 1		&   (1^7,3)  		& 1\\
(1,3,6) 	& 1		&     			&\\
\end{array}\]
The sum of the weights is 24, which is the same as the count of partitions with the altering sign with respect to their smallest part's parity.
\end{center}
\end{table}

The proof of Theorem~\ref{Ordinary_Partitions_Combinatorial_Weighted_Theorem} will be given as the combinatorial interpretation of the following analytic identity.

\begin{theorem}\label{Ordinary_Partitions_Analytic_Theorem}
\begin{equation}\label{Ordinary_Partitions_Analytic_Identity}
\sum_{n\geq 1} \frac{q^n}{1+q^n}\frac{\tiny 1}{(q;q)_{n-1}} = \sum_{n\geq 1} \frac{q^{n(n+1)/2}}{(q^2;q^2)_n (q^{n+1};q)_\infty}.
\end{equation}
\end{theorem}

\begin{proof}
Recall that $(0;q)_n =1$ for any integer $n\geq 0$. Also note that \begin{equation}\label{factorial_simplification}\frac{1+q}{1+q^n} = \frac{(-q;q)_{n-1}}{(-q^2;q)_{n-1}},\end{equation} for positive $n$. We start by writing the left--hand side of \eqref{Ordinary_Partitions_Analytic_Identity} as a $q$-hypergeometric function. Multiplying and dividing with $1+q$ and using \eqref{factorial_simplification}, shifting the sum with $n\mapsto n+1$, and finally grouping out $q/(1+q)$ yields \begin{equation}\label{Ordinary_2_Phi_1}
\sum_{n\geq 1} \frac{q^n}{1+q^n}\frac{\tiny 1}{(q;q)_{n-1}} = \frac{q}{1+q} {}_2\phi_1 \left(\genfrac{}{}{0pt}{}{0,-q}{-q^2};q,q\right).
\end{equation} 
We now apply the Jackson's transformation \eqref{Jackson_transformation} to \eqref{Ordinary_2_Phi_1}. This gives us
\begin{equation}\label{Ordinary_2_phi_2}
\frac{q}{1+q} {}_2\phi_1 \left(\genfrac{}{}{0pt}{}{0,-q}{-q^2};q,q\right) = \frac{q}{1+q} \frac{\tiny 1}{(q;q)_\infty}{}_2\phi_2 \left(\genfrac{}{}{0pt}{}{0,\, q}{-q^2,\, 0};q,-q^2\right).
\end{equation}
Distributing the front factor to each summand on the right-hand side of \eqref{Ordinary_2_phi_2}, doing the necessary simplifications, and finally shifting the summation index $n\mapsto n-1$ finishes the proof.
\end{proof}

Theorem~\ref{Ordinary_Partitions_Analytic_Theorem} is the analytical version of Theorem~\ref{Ordinary_Partitions_Combinatorial_Weighted_Theorem}. We will now move on to the generating function interpretations of both sides of \eqref{Ordinary_Partitions_Analytic_Identity}. This study will in-turn prove Theorem~\ref{Ordinary_Partitions_Combinatorial_Weighted_Theorem}.

We start with the left-hand side sum \begin{equation}\label{Ordinary_Left_s}\sum_{n\geq 1} \frac{q^n}{1+q^n}\frac{\tiny 1}{(q;q)_{n-1}},\end{equation} of \eqref{Ordinary_Partitions_Analytic_Identity}. For a positive integer $n$, the summand \begin{equation}\label{Ordinary_Analytic_left_weighted_bit}\frac{q^n}{1+q^n} = \sum_{k\geq 1} (-1)^{k+1} q^{nk} = q^n -q^{2n} +q^{3n} \dots\end{equation} is the generating function for the number of partitions of the form $(k^n)$ where the partition gets counted with the weight $+1$ if the part $k$ is odd and it gets counted with the weight $-1$ if the part $k$ is even. The factor \begin{equation}\label{Ordinary_Analytic_left_facorial_bit}\frac{\tiny 1}{(q;q)_{n-1}}\end{equation} is the generating function for the number of partitions into parts less than $n$. With conjugation in mind, another equivalent interpretation of \eqref{Ordinary_Analytic_left_facorial_bit} is that it is the generating function for the number of partitions into less than $n$ parts. 

We put the partitions counted by the factors in the summand into a single partition bijectively by part-by-part addition. For the same positive integer $n$, let $\pi_1$ be a partition counted by \eqref{Ordinary_Analytic_left_weighted_bit} and a partition $\pi_2$ counted by \eqref{Ordinary_Analytic_left_facorial_bit}. We know that $\pi_1 = (k^n)$ for some positive integer $k$. Starting from the largest part of $\pi_2$, we add a part of $\pi_2$ to a part of $\pi_1$ and put the outcome as a part of a new partition $\pi$. Recall that a part of a partition is a positive integer that is an element of that partition. The partition $\pi_2$ has less than $n$ parts. Therefore, there is at least one part of $\pi_1$ that does not get anything added to it. We add these leftover parts of $\pi_1$ to $\pi$ after the additions. This way we know that the new partition $\pi$ has exactly $n$ parts, where the smallest part is exactly $k$. This can be easily demonstrated using Ferrers diagrams in Table~\ref{Ferrers_adding}.

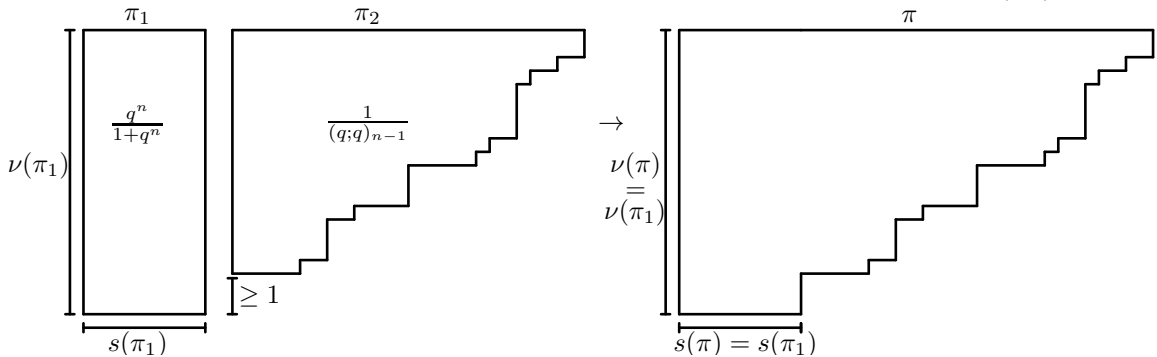
\begin{table}[htb]\caption{Demonstration of putting together partitions in the summand of \eqref{Ordinary_Left_s}}\label{Ferrers_adding}
\begin{center}\vspace{-.5cm}
\definecolor{cqcqcq}{rgb}{0.75,0.75,0.75}
\begin{tikzpicture}[line cap=round,line join=round,>=triangle 45,x=0.18cm,y=0.18cm]
\clip(-1.5,-2) rectangle (83.5,24);
\draw [line width=1pt] (4,22)-- (13,22);
\draw [line width=1pt] (13,22)-- (13,1);
\draw [line width=1pt] (13,1)-- (4,1);
\draw [line width=1pt] (4,1)-- (4,22);
\draw [line width=1pt] (15,22)-- (15,4);
\draw [line width=1pt] (41,22)-- (15,22);
\draw [line width=1pt] (15,4)-- (20,4);
\draw [line width=1pt] (20,4)-- (20,5);
\draw [line width=1pt] (20,5)-- (22,5);
\draw [line width=1pt] (22,8)-- (22,5);
\draw [line width=1pt] (22,8)-- (24,8);
\draw [line width=1pt] (24,8)-- (24,9);
\draw [line width=1pt] (24,9)-- (28,9);
\draw [line width=1pt] (28,9)-- (28,12);
\draw [line width=1pt] (28,12)-- (33,12);
\draw [line width=1pt] (33,12)-- (33,13);
\draw [line width=1pt] (33,13)-- (34,13);
\draw [line width=1pt] (34,13)-- (34,14);
\draw [line width=1pt] (34,14)-- (36,14);
\draw [line width=1pt] (36,14)-- (36,18);
\draw [line width=1pt] (36,18)-- (37,18);
\draw [line width=1pt] (37,18)-- (37,19);
\draw [line width=1pt] (37,19)-- (39,19);
\draw [line width=1pt] (39,19)-- (39,20);
\draw [line width=1pt] (39,20)-- (41,20);
\draw [line width=1pt] (41,20)-- (41,22);
\draw [line width=1pt] (48,22)-- (57,22);
\draw [line width=1pt] (57,1)-- (48,1);
\draw [line width=1pt] (48,1)-- (48,22);
\draw [line width=1pt] (83,22)-- (57,22);
\draw [line width=1pt] (57,4)-- (62,4);
\draw [line width=1pt] (62,4)-- (62,5);
\draw [line width=1pt] (62,5)-- (64,5);
\draw [line width=1pt] (64,8)-- (64,5);
\draw [line width=1pt] (64,8)-- (66,8);
\draw [line width=1pt] (66,8)-- (66,9);
\draw [line width=1pt] (66,9)-- (70,9);
\draw [line width=1pt] (70,9)-- (70,12);
\draw [line width=1pt] (70,12)-- (75,12);
\draw [line width=1pt] (75,12)-- (75,13);
\draw [line width=1pt] (75,13)-- (76,13);
\draw [line width=1pt] (76,13)-- (76,14);
\draw [line width=1pt] (76,14)-- (78,14);
\draw [line width=1pt] (78,14)-- (78,18);
\draw [line width=1pt] (78,18)-- (79,18);
\draw [line width=1pt] (79,18)-- (79,19);
\draw [line width=1pt] (79,19)-- (81,19);
\draw [line width=1pt] (81,19)-- (81,20);
\draw [line width=1pt] (81,20)-- (83,20);
\draw [line width=1pt] (83,20)-- (83,22);
\draw [line width=1pt] (57,1)-- (57,4);
\draw (8,15) node[anchor=center] {$ \frac{q^n}{1+q^n} $};
\draw (25,15) node[anchor=center] {$ \frac{\tiny 1}{(q;q)_{n-1}} $};
\draw (43,15) node[anchor=center] {$ \rightarrow $};
\draw [line width=1pt] (14.66,3.66)-- (15.33,3.66);
\draw [line width=1pt] (14.66,1)-- (15.33,1);
\draw [line width=1pt] (15,3.66)-- (15,1);
\draw (17.0,2.3) node[anchor=center] {$ \geq 1$};
\draw [line width=1pt] (2.66,22)-- (3.33,22);
\draw [line width=1pt] (2.66,1)-- (3.33,1);
\draw [line width=1pt] (3,22)-- (3,1);
\draw (0.7,12) node[anchor=center] {$\nu(\pi_1)$};
\draw (8,23) node[anchor=center] {$\pi_1$};
\draw (25,23) node[anchor=center] {$\pi_2$};
\draw (65,23) node[anchor=center] {$\pi$};
\draw [line width=1pt] (4,0.33)-- (4,-0.33);
\draw [line width=1pt] (4,0)-- (13,0);
\draw [line width=1pt] (13,0.33)-- (13,-0.33);
\draw [line width=1pt] (48,0.33)-- (48,-0.33);
\draw [line width=1pt] (48,0)-- (57,0);
\draw [line width=1pt] (57,0.33)-- (57,-0.33);
\draw (8,-1.25) node[anchor=center] {$s(\pi_1)$};
\draw (52.9,-1.25) node[anchor=center] {$s(\pi)=s(\pi_1)$};
\draw [line width=1pt] (57,0.33)-- (57,-0.33);
\draw [line width=1pt] (46.66,22)-- (47.33,22);
\draw [line width=1pt] (46.66,1)-- (47.33,1);
\draw [line width=1pt] (47,22)-- (47,1);
\draw (44.8,12) node[anchor=center] {$\nu(\pi)$};
\draw (44.8,10) node[anchor=center] {$=$};
\draw (44.75,8.5) node[anchor=center] {$\nu(\pi_1)$};
\end{tikzpicture}
\end{center}
\end{table}

Moreover, the partition $\pi$ gets counted with the weight $+1$ if the smallest part is odd and it gets counted with the weight $-1$ if the smallest part is even. The sum of all these terms gives us the generating function for the weighted count of ordinary partitions from $\U$. Hence,
\begin{equation}\label{Ordinary_left_revelation}
\sum_{n\geq 1} \frac{q^n}{1+q^n}\frac{\tiny 1}{(q;q)_{n-1}} = \sum_{\pi\in\U} (-1)^{s(\pi)+1} q^{|\pi|},
\end{equation} where $s(\pi)$ is the smallest part of the partition $\pi$.

The right-hand side summation \begin{equation}\label{Over_right_sum_s}\sum_{n\geq 1} \frac{q^{n(n+1)/2}}{(q^2;q^2)_n (q^{n+1};q)_\infty}\end{equation} of \eqref{Ordinary_Partitions_Analytic_Identity} can also be interpreted as a weighted count of partitions. For some positive integer $n$, the term $q^{n(n+1)/2}$ can be thought of as the generating function of the partition $\pi^*_1 = (1,2,3,4,\dots,n)$ where every part less than or equal to $n$ appears exactly one time. The factor \begin{equation}\label{Ordinary_right_q_2_factor}
\frac{\tiny 1}{(q^2;q^2)_n}
\end{equation} is the generating function for partitions into parts $\leq n$ where every part appears with an even frequency. Let $\pi^*_2$ be a partition counted by \eqref{Ordinary_right_q_2_factor}. By adding the frequencies of $\pi^*_1$ and $\pi^*_2$ we get another partition \[\pi^* = (1^{f_1},2^{f_2},\dots,n^{f_n}),\] where all $f_i\equiv 1\mod{\tiny 2}$. The quotient \begin{equation}\label{Ordinary_right_q_n_1_factor}
\frac{\tiny 1}{(q^{n+1};q)_\infty}
\end{equation} is the generating function for the number of partitions into parts $> n$. Therefore, for a partition $\pi'$ that is counted by \eqref{Ordinary_right_q_n_1_factor} one can put together $\pi^*$ and $\pi'$ without the need of adding any frequencies. Call the outcome partition of merging $\pi^*$ and $\pi'$, $\pi$. 

With this interpretation, the partitions counted by \eqref{Over_right_sum_s} have the frequency restriction that $f_1(\pi)\equiv {1\mod {\tiny 2}}$. Also, let $i$ be the first positive integer where $f_i(\pi)$ is even (maybe zero). It is obvious that the partition $\pi$ might be the final outcome of the merging procedure explained above for any summand in \eqref{Over_right_sum_s} as long as the index of the summand is $< i$. Therefore, the partition $\pi$ is weighted by the number of the parts in its initial chain of odd frequencies of parts. This proves \begin{equation}\label{Ordinary_revelation_right_side}\sum_{n\geq 1} \frac{q^{n(n+1)/2}}{(q^2;q^2)_n (q^{n+1};q)_\infty} = \sum_{\pi\in\U^*} t(\pi) q^{|\pi|},\end{equation} where $t(\pi)$ is as defined in Theorem~\ref{Ordinary_Partitions_Combinatorial_Weighted_Theorem}. The identities \eqref{Ordinary_left_revelation} and \eqref{Ordinary_revelation_right_side} together prove Theorem~\ref{Ordinary_Partitions_Combinatorial_Weighted_Theorem}.

Now we move on to another analytical identity similar to \eqref{Ordinary_Partitions_Analytic_Identity}. This identity will later prove a weighted partition identity for overpartitions.

\begin{theorem}\label{Over_Analytical_Theorem}
\begin{equation}\label{Over_Analytical_Identity} \sum_{n\geq 1} \frac{2q^n}{1+q^n} \frac{(-q;q)_{n-1}}{(q;q)_{n-1}} = \sum_{n\geq 0} \frac{q^{n(n+1)/2}}{(q;q)_n}\frac{2q^{n+1}}{1-q^{2(n+1)}} \frac{(-q^{n+2};q)_\infty}{(q^{n+2};q)_\infty}
\end{equation}
\end{theorem}

\begin{proof}Multiply and divide the left-hand side of \eqref{Over_Analytical_Identity} by $(1+q)$, use \eqref{factorial_simplification}, and write it as a $q$-hypergeometric series: \begin{equation}\label{Over_left_to_2_phi_1}
\sum_{n\geq 1} \frac{2q^n}{1+q^n} \frac{(-q;q)_{n-1}}{(q;q)_{n-1}} = \frac{2q}{1+q} {}_2\phi_1 \left(\genfrac{}{}{0pt}{}{-q,-q}{-q^2};q,q\right).
\end{equation} Now we apply the transformation \eqref{Jackson_transformation} to \eqref{Over_left_to_2_phi_1}. This yields, 
\begin{equation}\label{Over_Jackson_transformed}
\frac{2q}{1+q} {}_2\phi_1 \left(\genfrac{}{}{0pt}{}{-q,-q}{-q^2};q,q\right) = \frac{2q}{1+q}\frac{(-q^2;q)_\infty}{(q;q)_\infty} {}_2\phi_2 \left(\genfrac{}{}{0pt}{}{-q,q}{-q^2,-q^2};q,-q^2\right).
\end{equation}Distributing the front factor to each summand, doing the necessary simplifications, and regrouping terms shows that the right-hand sides of identities \eqref{Over_Analytical_Identity} and \eqref{Over_Jackson_transformed} are equal.
\end{proof}

The combinatorial interpretation of \eqref{Over_Analytical_Identity} is similar to the one of \eqref{Ordinary_Partitions_Analytic_Identity}. The left-hand side sum \[\sum_{n\geq 1} \frac{2q^n}{1+q^n} \frac{(-q;q)_{n-1}}{(q;q)_{n-1}}\] of \eqref{Over_Analytical_Identity}. For a given $n$ the summand factor \[\frac{2q^n}{1+q^n}\] is the generating function of the number of overpartitions into exactly $n$ parts of the same size, where the partitions are counted with weight $+1$ if the part is odd and with $-1$ if the part is even. In other words, it is the generating function for the number of partitions $(k^n)$ and $(\bar{k}^n)$ for any integer $k\geq 1$, where these partitions are counted with the weight $(-1)^{k+1}$. The other factor \begin{equation}\label{overpartitions_factor_less_than_n}\frac{(-q;q)_{n-1}}{(q;q)_{n-1}},\end{equation} (by \eqref{weights_of_overpartitions}) is the generating function for the number of overpartitions with strictly less than $n$ parts. As we did in the proof of Theorem~\ref{Ordinary_Partitions_Combinatorial_Weighted_Theorem}, we put the parts of these partitions together. This part-by-part addition gives an overpartition in exactly $n$ parts with the smallest part $k$. And coming from the first factor we count these partitions with weight $+1$ if the smallest part $k$ is odd and with weight $-1$ if $k$  is even. Hence,
\begin{equation}\label{Over_left_combinatorial}
\sum_{n\geq 1} \frac{2q^n}{1+q^n} \frac{(-q;q)_{n-1}}{(q;q)_{n-1}} = \sum_{\pi\in\mathcal{O}} (-1)^{s(\pi)+1}q^{|\pi|}.
\end{equation}

The right-hand side of \eqref{Over_Analytical_Identity} can be interpreted  in a way  similar to that of \eqref{Over_right_sum_s}. For some non-negative integer $n$, the factor \begin{equation}\label{Over_right_factor_1}\frac{q^{n(n+1)/2}}{(q;q)_n}\end{equation} is the generating function for number of partitions of the type $(1^{f_1},2^{f_2},\dots,n^{f_n})$, where $f_i\geq 1$ for all $1\leq i\leq n$, as $n(n+1)/2 = 1+2+\dots+n$. The rest of the factors \begin{equation}\label{Over_right_factor_2_3}\frac{2q^{n+1}}{1-q^{2(n+1)}} \frac{(-q^{n+2};q)_\infty}{(q^{n+2};q)_\infty}\end{equation} can be interpreted as the generating function for the number of overpartitions where the smallest part (which definitely appears in the partition) is $n+1$ and that part has an odd frequency. 

There is no overlapping in the size of the parts in the partitions counted by \eqref{Over_right_factor_1} and \eqref{Over_right_factor_2_3} for a fixed $n$. One can merge these partitions into a single partition without any need of non-trivial addition of frequencies. On the other hand, an outcome overpartition may be coming from different merged couples of partitions/overpartitions. Given an outcome overpartition, there is no clean cut point that would indicate where the overpartition counted by \eqref{Over_right_factor_2_3} started. The only indication is the odd frequency of the smallest part of overpartitions. Also, we know that every part below the smallest part of overpartition in the combined partition is coming from a partition counted by the generating function \eqref{Over_right_factor_1}. In particular, 1 appears as a part in any outcome of this merging process. Therefore, we need to keep account of all these possible connection points when we are finding the count of a partition coming from the right-hand side of \eqref{Over_Analytical_Identity}. By going through only the odd frequencies in a given partition and counting the number of larger parts with the overpartition weights, we can find the total count of combinations that would yield the same merged overpartition images. 

Given a partition $\pi$, let $m(\pi)$ be the smallest positive integer that is not a part of $\pi$. Let $\nu_d(\pi,n)$ be the number of different parts $\geq n$ in partition $\pi$. Let 
\begin{equation}\label{truth_function}\chi(\textit{statement})=\left\{\begin{array}{cc}
1, &\text{if the \textit{statement} is true},\\
0, &\text{otherwise},
\end{array}\right. \end{equation}
be the \textit{truth} function. 

Then, the right-hand side of \eqref{Over_Analytical_Identity} can be written as a weighted count of partitions as \begin{equation}\label{Over_right_combinatorial}
\sum_{n\geq 0} \frac{q^{n(n+1)/2}}{(q;q)_n}\frac{2q^{n+1}}{1-q^{2(n+1)}} \frac{(-q^{n+2};q)_\infty}{(q^{n+2};q)_\infty} = \sum_{\pi\in\U} \tau(\pi) q^{|\pi|},
\end{equation} where \begin{equation}\label{Over_Tau_Def}
\tau(\pi) = \sum_{i=1}^{m(\pi)} \chi(f_{i}\equiv1(\text{mod }2)) 2^{\nu_d(\pi,i)}.
\end{equation} 

This study proves the combinatorial version of Theorem~\ref{Over_Analytical_Theorem}. We put \eqref{Over_left_combinatorial} and \eqref{Over_right_combinatorial} together, and get the following theorem. 

\begin{theorem}\label{Over_Weighted_Theorem}
\begin{equation}\label{Over_Weighted_Identity}\sum_{\pi\in\mathcal{O}} (-1)^{s(\pi)+1}q^{|\pi|} = \sum_{\pi\in\U} \tau(\pi) q^{|\pi|},
\end{equation} where $\tau(\pi)$ is defined as in \eqref{Over_Tau_Def}.
\end{theorem} 

There are 100 overpartitions of $8$. There are 18 overpartitions of 8 with an even smallest part. Hence, in the weighted count of the left-hand side of \eqref{Over_Weighted_Identity} the coefficient of $q^8$ term is $100-2\cdot18 = 64$. We exemplify the right-hand side weights of Theorem~\ref{Over_Weighted_Theorem} for the same norm in Table~\ref{Table_Over_Partitions_tau_weight}.

\begin{table}[htb]\caption{Example of Theorem~\ref{Over_Weighted_Theorem} with $|\pi|=8$.}\label{Table_Over_Partitions_tau_weight}
\begin{center}\vspace{-.5cm}
\[\begin{array}{cc|cc}
\pi\in\U  &  \tau(\pi) &  \pi\in\U  &  \tau(\pi) \\[-2ex]& &  \\
(1^3,2,3) 	& 8+4+2=14		&   (1^2,2,4) 		& 4		\\
(1,2,5)		& 8+4=12		&   (1^3,5) 		&  4		\\
(1,2^2,3)	& 8+2=10		&   (1,7) 		&	 4	\\
(1,3,4) 	& 8		&   (1^6,2)	 	&	2	\\
(1^5,3) 	& 4		&   (1^2,2^3)  		& 2	\\
\end{array}\]
The sum of the weights is 64, which is the same as the count of overpartitions with the alternating sign with respect to their smallest part's parity.
\end{center}
\end{table}

\section{A Weighted Identity with respect to the Smallest Part and the Number of Parts of a Partition in relation with Sums of Squares}\label{Section_s_v}

We start with a short proof of an analytic identity.
\begin{lemma}\label{RAMA_lemma}
\begin{equation}\label{RAMA_lemma_equation}
\sum_{n\geq 1} \frac{(-1)^n q^{n(n+1)/2}}{(1+q^n)(q;q)_n} = \sum_{n\geq 1} (-1)^n q^{n^2}.
\end{equation}
\end{lemma}

\begin{proof} It is easy to see that \begin{equation}\label{apply_q_Gauss_to_this}
1+2\sum_{n\geq 1} \frac{(-1)^n q^{n(n+1)/2}}{(1+q^n)(q;q)_n} = \lim_{\rho\rightarrow\infty} {}_2\phi_1 \left( \genfrac{}{}{0pt}{}{-1,\ \rho q}{-q};q, 1/\rho \right) = \frac{(q;q)_\infty}{(-q;q)_\infty},
\end{equation}
where we used $q$-Gauss sum \eqref{q_Gauss}. Rewiriting the sum in \eqref{Theta_phi_prod} as \begin{equation}\label{Theta_squares_rewritten}1+2\sum_{n\geq 1} (-1)^n q^{n^2}\end{equation} proves the claim.
\end{proof}

The identity \eqref{RAMA_lemma_equation} is a special case of a more general identity of Ramanujan \cite[E. 1.6.2, p. 25]{LostNotebook_2} which even has a combinatorial proof \cite{BerndtKimYee}. But, more relevant to this paper, Alladi \cite[Thm 2, p. 330]{Alladi_RAMA_identity} is the first one to give a combinatorial interpretation to the left-hand side of Therorem~\ref{RAMA_lemma} in the spirit of the Euler pentagonal number theorem. In his study, he interpreted the left-hand side sum as the number of partitions into distinct parts with smallest part being odd weighted with $+1$ or $-1$ depending on the number of parts of the partition being even or odd, respectively. In our notations:
\begin{theorem}[Alladi, 2009]\label{Alladi_squares_THM}Let $N$ be a positive integer. Then,\begin{equation}
\sum_{\substack{\pi\in\mathcal{D}_o,\\ |\pi|=N}} (-1)^{\nu(\pi)}  = (-1)^N \chi(N = \square\, ),
\end{equation}
where $\mathcal{D}_o$ is the set of non-empty partitions into distinct parts where the smallest part is odd, $\chi$ is as defined in \eqref{truth_function}, and $\square$ represents the statement ``a perfect integer square."
\end{theorem}

It is easy to check that \begin{equation*}
|\pi|\equiv \nu_o(\pi) \mod{2}
,\end{equation*} for any partition $\pi$. Hence, \begin{equation}\label{parity_equivalence}
\nu(\pi) - |\pi| \equiv \nu_e(\pi) \mod{2}.
\end{equation} This enables us to rewrite Theorem~\ref{Alladi_squares_THM} as in \cite{BessenrodtPak}.

\begin{theorem}[Bessenrodt, Pak, 2004]\label{Alladi_squares_simplified}Let $N$ be a positive integer. Then,\begin{equation}
\sum_{\substack{\pi\in\mathcal{D}_o,\\ |\pi|=N}} (-1)^{\nu_e(\pi)}  = \chi(N = \square\, ).
\end{equation}
\end{theorem}

There, they also discussed a refinement of Theorem~\ref{Alladi_squares_simplified}. 

\begin{theorem}[Bessenrodt, Pak, 2004]\label{Bessenrodt_Pak_THM}\begin{equation}
\sum_{\substack{\pi\in\mathcal{D}_o,\\ |\pi|=N,\\ \nu_o(\pi)=k}} (-1)^{\nu_e(\pi)} =  \chi(N=k^2).\end{equation}
\end{theorem}

Theorems~\ref{Alladi_squares_THM}$\,$--$\,$\ref{Bessenrodt_Pak_THM} connect the weighted count of the partitions into distinct parts, where the smallest part is necessarily odd and the number of representation of integer as a perfect square. Our next theorem will be connecting the weighted count of partitions and the number of representations of a number as a sum of two squares.

\begin{theorem}\label{Over_S_V_analytic_theorem}\begin{equation}
\label{Over_S_V_analytic_identity}
\sum_{n\geq 1} \frac{(-1)^n 2 q^n}{1+q^n}\frac{(-q;q)_{n-1}}{(q;q)_{n-1}} = \varphi(-q)^2 - \varphi(-q)
\end{equation}
\end{theorem}

\begin{proof} Similar to the proof of Theorem~\ref{Ordinary_Partitions_Analytic_Theorem} we would like to write the left-hand side of \eqref{Over_S_V_analytic_identity} as a hypergeometric function first. On the left-hand side of \eqref{Over_S_V_analytic_identity} we multiply and divide the summand by $(1+q)$, use \eqref{factorial_simplification}, factor out the terms $-2q/(1+q)$, and finally shift the summation variable $n\mapsto n+1$ to write the expression as a ${}_2\phi_1$ hypergeometric series. Applying the Jackson's transformation \eqref{Jackson_transformation} to this expression yields \begin{equation}
\frac{-2q}{1+q} {}_2\phi_1 \left(\genfrac{}{}{0pt}{}{-q,-q}{-q^2};q,-q\right) = \frac{-2q}{1+q}\frac{(q^2;q)_\infty}{(-q;q)_\infty} {}_2\phi_2 \left(\genfrac{}{}{0pt}{}{-q,q}{-q^2,q^2};q,q^2\right).
\end{equation}
Writing the ${}_2\phi_2$ explicitly, distributing the factor q/(1+q), performing the simple cancellations, shifting the summation variable $n\mapsto n-1$ and multiplying and dividing with $1-q$ we get \begin{equation}\label{apply_RAMA_to_this}
\frac{-2q}{1+q}\frac{(q^2;q)_\infty}{(-q;q)_\infty} {}_2\phi_2 \left(\genfrac{}{}{0pt}{}{-q,q}{-q^2,q^2};q,q^2\right) = 2\frac{(q;q)_\infty}{(-q;q)_\infty} \sum_{n\geq 1} \frac{(-1)^{n} q^{n(n+1)/2}}{(1+q^{n})(q;q)_{n}}.
\end{equation}
Applying Lemma~\ref{RAMA_lemma} to the right-hand side of \eqref{apply_RAMA_to_this} and rewriting the identity we see that 
\begin{equation}\label{last_line_of_over_s_v_proof}
\sum_{n\geq 1} \frac{(-1)^n 2 q^n}{1+q^n}\frac{(-q;q)_{n-1}}{(q;q)_{n-1}} =  2\frac{(q;q)_\infty}{(-q;q)_\infty} \left(-1 + \sum_{n\geq 0} (-1)^nq^{n^2} \right).
\end{equation} Using observations \eqref{Theta_squares_rewritten} and \eqref{Theta_phi_prod} on the right-hand side of \eqref{last_line_of_over_s_v_proof} we complete the proof.
\end{proof}

The combinatorial interpretation of Theorem~\ref{Over_S_V_analytic_theorem} combines a weighted partition count with a representation of numbers by sum of two squares. Moreover, we can provide an explicit formula for the weighted count of partitions with respect to the norm.

Let $n$ be a positive integer. The summand \[\frac{(-1)^n 2 q^n}{1+q^n}\] of \eqref{Over_S_V_analytic_identity} is the generating function for the number of partitions of the form $(k^n)$ (keeping \eqref{Ordinary_Analytic_left_weighted_bit} in mind) gets counted with the weight $ (-1)^{k+n+1}2$. Here it should be noted that $k$ is the smallest part and $n$ is the number of parts of this partition. After the needed addition of partitions (similar to the ones we did for Theorems \eqref{Ordinary_Partitions_Combinatorial_Weighted_Theorem} and \eqref{Over_Weighted_Theorem}) these two variables are going to stay the same for the outcome partition. To have a uniform notation, recall that $\nu(\pi)$ denotes the number of parts, and $\nu_d(\pi)$ is the number of different parts of a partition $\pi$. The second summand \begin{equation}\tag{\ref{overpartitions_factor_less_than_n}} \frac{(-q;q)_{n-1}}{(q;q)_{n-1}} \end{equation} that appears in \eqref{Over_S_V_analytic_identity} is the generating function for the number of overpartitions into strictly less than $n$ parts as mentioned before. We know that this is the same as counting the number of ordinary partitions $\pi$ in less than $n$ parts counted with the weight $2^{\nu_d(\pi)}$ by \eqref{over_connect_regular_abstract}. 

Putting together the partition $\pi_1=(k^n)$ and a partition $\pi_2$ counted by \eqref{overpartitions_factor_less_than_n} (similar to the way we did in Table~\ref{Ferrers_adding}) gives us an outcome overpartition $\pi$. The partition $\pi$ has the properties $s(\pi)=k$, $\nu(\pi)=n$, and $\nu_d(\pi) = \nu_d(\pi_1)+\nu_d(\pi_2)=\nu_d(\pi_2)+1$. This partition is counted with the weight\begin{equation}
\label{weight_for_s_v} \omega(\pi):=(-1)^{s(\pi)+\nu(\pi)+1} 2^{\nu_d(\pi)},
\end{equation} (the multiplication of weights of $\pi$'s  generators) by the right-hand side of \eqref{Over_S_V_analytic_identity}. This proves \begin{equation}\label{over_s_v_left_side}
\sum_{n\geq 1} \frac{(-1)^n 2 q^n}{1+q^n}\frac{(-q;q)_{n-1}}{(q;q)_{n-1}} = \sum_{\pi\in\U} \omega(\pi) q^{|\pi|}.
\end{equation}

On the other side of the equation \eqref{Over_S_V_analytic_identity} we have the difference of two theta series. The summation of \eqref{Theta_phi_prod} is enough to see that \begin{equation}\label{varphi_explanation}\varphi(-q)^2 - \varphi(-q) = \sum_{x,y \in \mathbb{Z}} (-1)^{x+y} q^{x^2+y^2} - \sum_{n=-\infty}^\infty (-1)^n q^{n^2}.\end{equation} 
Let $r_2(N)$ be the number of representations of $N$ as a sum of two squares. Any positive integer $N$ has the unique prime factorization \[N = 2^e\prod_{i\geq 1} p_i^{v_i} \prod_{j\geq 1} q_j^{w_j},\] where $e$, $v_i$, and $w_j$ are non-negative integers, and $p_i$ and $q_j$ are primes 1 and 3 mod 4, respectively. It is known \cite[Thm 14.13, p. 572]{Rosen} that \begin{equation}\label{r_2}
r_2(N) = 4 \prod_{i\geq 1} (1+v_i) \prod_{j\geq 1} \frac{1+(-1)^{w_j}}{\tiny 2}.
\end{equation}
Writing the first series organized with respect to $r_2$, rewriting the second series, and finally cancelling the constant terms of both series we get \begin{equation}\label{collect_the_squares}
\sum_{x,y \in \mathbb{Z}} (-1)^{x+y} q^{x^2+y^2} - \sum_{n=-\infty}^\infty (-1)^nq^{n^2} = \sum_{N\geq 1} (-1)^N r_2(N)q^N - 2\sum_{n\geq 1} (-1)^n q^{n^2}.
\end{equation} On the right-hand side of \eqref{collect_the_squares} one can collect the terms with respect to the exponents of $q$. Writing the two series together with the use of a truth function and comparing  \eqref{varphi_explanation} and \eqref{collect_the_squares} yields the identity \begin{equation}\label{over_s_v_right_side}
\varphi(-q)^2 - \varphi(-q) =  \sum_{N\geq 1} (-1)^N ( r_2(N) - 2\chi(N=\square)\, )q^N, 
\end{equation} where $\chi$ is defined as in \eqref{truth_function} and $\square$ represents ``a perfect integer square."

Now we put the right-hand sides of \eqref{weight_for_s_v}, \eqref{over_s_v_left_side} and \eqref{over_s_v_right_side} together and get an explicit expression for the sum of weights $\omega(\pi)$ of partitions for a fixed positive norm $N$:
\begin{equation}\label{Explicit_weights}
\sum_{\substack{\pi\in\U,\\ |\pi|=N}} (-1)^{s(\pi)+\nu(\pi)+1} 2^{\nu_d(\pi)} = (-1)^N ( r_2(N) - 2\chi(N=\square)\, ).\end{equation}

We can employ the observation \eqref{parity_equivalence} to simplify \eqref{Explicit_weights}.

\begin{theorem}\label{Explicit_weights_theorem}\begin{equation}
\sum_{\substack{\pi\in\U,\\ |\pi|=N}} \omega^*(\pi) = r_2(N) - 2\chi(N=\square),\end{equation} where \[\omega^*(\pi) = (-1)^{s(\pi)+\nu_e(\pi)+1} 2^{\nu_d(\pi)}. \]
\end{theorem}

Two examples of Theorem~\ref{Explicit_weights_theorem} are given in Table~\ref{Table_Over_S_V}.

\begin{table}[htb]\caption{Examples of Theorem~\ref{Explicit_weights_theorem} with $|\pi|=4$ and $5$.}\label{Table_Over_S_V}
\begin{center}\vspace{-.5cm}
\[\begin{array}{ccc||cc}
&\pi\in\U,\ |\pi|=4  &  \omega^*(\pi) &  \pi\in\U,\ |\pi|=5 &  \omega^*(\pi) \\[-2ex]& &  \\
&(4) 	& 	2	&   (5)			& 	2		\\
&(2^2) 	&	-2	&	(2,3)		&	2^2	\\
&(1,3)	& 	2^2	&   (1,4) 		&  	-2^2		\\
&(1^2,2)	& -2^2	&   (1^2,3) 	&	2^2 		\\
&(1^4) 	& 	2	&   (1,2^2)	 	&	2^2		\\
&	 	& 		&   (1^3,2)  	& 	-2^2		\\
&		&		&	(1^5)		&	2		\\ \hline
\text{Total:}&		&  2 & & 8\\
\end{array}\]
and the explicit formula of \eqref{Explicit_weights} suggests:
\begin{align*}
&(r_2(4) -2 \cdot 1) = 4 -2 = 2, \\
&(r_2(5) - 2\cdot 0) = 8.
\end{align*}
\end{center}
\end{table}

Another equivalent statement of Theorem~\ref{Explicit_weights_theorem} can be given over the set of overpartitions by evaluating \eqref{weight_for_s_v} and \eqref{over_connect_regular_abstract}.

\begin{theorem}\label{Over_written_squares_THM}\begin{equation*}
\sum_{\substack{\pi\in\mathcal{O},\\ |\pi|=N}} (-1)^{s(\pi)+\nu_e(\pi)+1} =  r_2(N) - 2\chi(N=\square).\end{equation*}
\end{theorem}


\section{Some Weighted Identities for Partitions with Distinct Even Parts}\label{Section_P_ped}

Let $\mathcal{P}$ denote the set of non-empty partitions with distinct even parts. A partition $\pi\in \P$ may still have repeated odd parts. This set has been studied before in \cite[$\S$ 5]{Alladi_Lebesgue}, \cite{Andrews_DistinctEvens} and \cite{Andrews_Hirschhorn_Sellers}. 


We start with the analytic identity:
\begin{theorem}\label{S_function_for_P_ped_THM}
\begin{equation}\label{S_function_for_P_ped_EQN}
\sum_{n\geq 1} \frac{(-1)^n q^{2n}}{1-q^{2n}}\frac{(-q;q^2)_{n-1}}{(q^2,q^2)_{n-1}}q^{n-1} = \psi(-q) - \frac{\tiny 1}{1+q}.
\end{equation}
\end{theorem}

\begin{proof} We multiply both sides of \eqref{S_function_for_P_ped_EQN} with $1+q$ and add 1. The resulting identity becomes a special case of the $q$-binomial theorem \eqref{q_binomial} with $(a,q,z) = (-1/q,q^2,-q^3)$ provided that we use \eqref{Theta_psi_product} with $q\mapsto -q$.
\end{proof}

The combinatorial interpretation of the left-hand side summand \begin{equation}\label{P_ped_left_S} \frac{(-1)^n q^{2n}}{1-q^{2n}}\frac{(-q;q^2)_{n-1}}{(q^2,q^2)_{n-1}}q^{n-1}\end{equation} for some positive $n$ is really similar to the previous constructions. The main difference is the use of 2-modular Ferrers diagrams, which has been introduced in Section~\ref{Section_def}, instead. We will be following similar steps that we followed in finding the combinatorial interpretation of Theorem~\ref{Ordinary_Partitions_Analytic_Theorem}.

Let $n$ be a fixed positive integer. The factor \begin{equation}\label{P_ped_geometric_series_factor}
\frac{(-1)^n q^{2n}}{1-q^{2n}} = \sum_{k\geq 1} (-1)^n q^{2kn}
\end{equation}
is the generating function of partitions of the type $\pi_1 = (\,(2k)^n\,)$ for some positive integer $k$, where these partitions get counted with a weight $+1$ if the number of parts of the partition $n$ is even and with $-1$ if $n$ is odd. The second factor \begin{equation}\label{P_ped_with_distinct_odd_factor}
\frac{(-q;q^2)_{n-1}}{(q^2,q^2)_{n-1}}
\end{equation} is the generating function for the number of partitions with distinct odd parts $\leq 2n-2$. We can express these partitions in 2-modular Ferrers diagrams and take their conjugates. The outcome would show that the same factor is the generating function for the number of partitions $\pi_2$ with distinct odd parts where the number of parts is $< n$. Finally, the term $q^{n-1}$ can be thought as the generating function of the partitions $\pi_3=(1^{n-1})$.

We would like to add the partitions $\pi_1$, $\pi_2$, and $\pi_3$ to make up a new partition. This will be done similar to the example of Table~\ref{Ferrers_adding}. We start by putting partitions $\pi_1$, $\pi_2$, and $\pi_3$ and add them up row-wise. When doing so, the possible boxes filled with 1's coming from $\pi_2$ are combined with the $1$'s of $\pi_3$ and turned into a row ending of a box with a $2$ in it. There being $n-1$ parts in $\pi_3$ and the row-wise addition of these partitions also makes sure that the outcome partition is a partition $\pi$ with distinct even parts where the smallest part is necessarily even. An illustration is given in Table~\ref{2Ferrer_addition}.

\begin{table}[htb]\caption{Demonstration of putting together partitions in the summand of \eqref{P_ped_left_S}}\label{2Ferrer_addition}
\begin{center}\vspace{-.5cm}
\definecolor{cqcqcq}{rgb}{0.75,0.75,0.75}
\begin{tikzpicture}[line cap=round,line join=round,>=triangle 45,x=0.177cm,y=0.177cm]
\clip(0.4,-3.5) rectangle (88,26.5);
\draw [line width=1pt] (5,22)-- (14,22);
\draw [line width=1pt] (14,22)-- (14,1);
\draw [line width=1pt] (14,1)-- (5,1);
\draw [line width=1pt] (5,1)-- (5,22);
\draw [line width=1pt] (15,22)-- (15,4);
\draw [line width=1pt] (41,22)-- (15,22);
\draw [line width=1pt] (15,4)-- (20,4);
\draw [line width=1pt] (20,4)-- (20,6);
\draw [line width=1pt] (20,6)-- (22,6);
\draw [line width=1pt] (22,8)-- (22,6);
\draw [line width=1pt] (22,8)-- (24,8);
\draw [line width=1pt] (24,8)-- (24,9);
\draw [line width=1pt] (24,9)-- (28,9);
\draw [line width=1pt] (28,9)-- (28,12);
\draw [line width=1pt] (28,12)-- (33,12);
\draw [line width=1pt] (33,12)-- (33,13);
\draw [line width=1pt] (33,13)-- (34,13);
\draw [line width=1pt] (34,13)-- (34,14);
\draw [line width=1pt] (34,14)-- (36,14);
\draw [line width=1pt] (36,14)-- (36,18);
\draw [line width=1pt] (36,18)-- (37,18);
\draw [line width=1pt] (37,18)-- (37,19);
\draw [line width=1pt] (37,19)-- (39,19);
\draw [line width=1pt] (39,19)-- (39,20);
\draw [line width=1pt] (39,20)-- (41,20);
\draw [line width=1pt] (41,20)-- (41,22);
\draw [line width=1pt] (3.66,22)-- (4.33,22);
\draw [line width=1pt] (3.66,1)-- (4.33,1);
\draw [line width=1pt] (4,22)-- (4,1);
\draw [line width=1pt] (5,0.33)-- (5,-0.33);
\draw [line width=1pt] (5,0)-- (14,0);
\draw [line width=1pt] (14,0.33)-- (14,-0.33);
\draw (5.6,21.3) node[anchor=center] {\tiny 2};
\draw (5.6,20.3) node[anchor=center] {\tiny 2};
\draw (6.6,21.3) node[anchor=center] {\tiny 2};
\draw (8.3,21.1) node[anchor=center] {\tiny \dots};
\draw (5.6,19.1) node[anchor=center] {\tiny \vdots};
\draw (13.5,1.6) node[anchor=center] {\tiny 2};
\draw (15.6,21.3) node[anchor=center] {\tiny 2};
\draw (17.3,21.1) node[anchor=center] {\tiny \dots};
\draw (15.6,20.1) node[anchor=center] {\tiny \vdots};
\draw [line width=1pt](42,22)-- (43,22);
\draw [line width=1pt](43,22)-- (43,2);
\draw [line width=1pt](43,2)-- (42,2);
\draw [line width=1pt](42,22)-- (42,2);
\draw [line width=1pt](44,22)-- (44,2);
\draw [line width=1pt](43.66,22)-- (44.33,22);
\draw [line width=1pt](43.66,2)-- (44.33,2);
\draw [line width=1pt](44,2)-- (44,22);
\draw (42.5,21.3) node[anchor=center] {\tiny 1};
\draw (42.5,20.3) node[anchor=center] {\tiny 1};
\draw (42.5,19.3) node[anchor=center] {\tiny \vdots};
\draw (42.5,3.4) node[anchor=center] {\tiny 1};
\draw (42.5,2.4) node[anchor=center] {\tiny 1};
\draw [line width=1pt] (52,22)-- (61,22);
\draw [line width=1pt] (61,1)-- (52,1);
\draw [line width=1pt] (52,1)-- (52,22);
\draw [line width=1pt] (87,22)-- (61,22);
\draw [line width=1pt] (67,4)-- (67,6);
\draw [line width=1pt] (67,6)-- (69,6);
\draw [line width=1pt] (69,8)-- (69,6);
\draw [line width=1pt] (69,8)-- (70,8);
\draw [line width=1pt] (70,8)-- (70,9);
\draw [line width=1pt] (70,9)-- (75,9);
\draw [line width=1pt] (75,9)-- (75,12);
\draw [line width=1pt] (75,12)-- (79,12);
\draw [line width=1pt] (79,12)-- (79,13);
\draw [line width=1pt] (79,13)-- (81,13);
\draw [line width=1pt] (81,13)-- (81,14);
\draw [line width=1pt] (81,14)-- (82,14);
\draw [line width=1pt] (84,18)-- (84,19);
\draw [line width=1pt] (84,19)-- (86,19);
\draw [line width=1pt] (86,19)-- (86,20);
\draw [line width=1pt] (86,20)-- (87,20);
\draw (40.5,21.3) node[anchor=center] {\tiny 2};
\draw (40.5,20.4) node[anchor=center] {\tiny 1};
\draw [line width=1pt] (82,14)-- (82,15);
\draw [line width=1pt] (82,15)-- (83,15);
\draw [line width=1pt] (83,15)-- (83,18);
\draw [line width=1pt] (83,18)-- (84,18);
\draw [line width=1pt] (88,21)-- (88,22);
\draw [line width=1pt] (88,22)-- (87,22);
\draw [line width=1pt] (88,21)-- (87,21);
\draw [line width=1pt] (87,21)-- (87,20);
\draw (60.4,1.6) node[anchor=center] {\tiny 2};
\draw [line width=1pt] (61,1)-- (61,2);
\draw [line width=1pt] (61,2)-- (62,2);
\draw [line width=1pt] (62,4)-- (62,2);
\draw [line width=1pt] (62,4)-- (67,4);
\draw (87.4,21.4) node[anchor=center] {\tiny 1};
\draw (86.4,20.4) node[anchor=center] {\tiny 2};
\draw (85.4,19.4) node[anchor=center] {\tiny 1};
\draw (83.4,18.4) node[anchor=center] {\tiny 1};
\draw (82.4,17.4) node[anchor=center] {\tiny 1};
\draw (82.4,16.4) node[anchor=center] {\tiny 1};
\draw (82.4,15.4) node[anchor=center] {\tiny 1};
\draw (81.4,14.4) node[anchor=center] {\tiny 2};
\draw (61.4,2.4) node[anchor=center] {\tiny 1};
\draw (61.4,3.4) node[anchor=center] {\tiny 1};
\draw (38.4,19.4) node[anchor=center] {\tiny 2};
\draw (36.4,18.4) node[anchor=center] {\tiny 2};
\draw (35.4,17.4) node[anchor=center] {\tiny 2};
\draw (35.4,16.4) node[anchor=center] {\tiny 2};
\draw (35.4,15.4) node[anchor=center] {\tiny 2};
\draw (35.4,14.4) node[anchor=center] {\tiny 1};
\draw (48,13.5) node[anchor=center] {$\rightarrow$};
\draw [line width=1pt] (52,0.33)-- (52,-0.33);
\draw [line width=1pt] (52,0)-- (61,0);
\draw [line width=1pt] (61,0.33)-- (61,-0.33);
\draw [line width=1pt](50.66,22)-- (51.33,22);
\draw [line width=1pt](50.66,1)-- (51.33,1);
\draw [line width=1pt](51,1)-- (51,22);
\draw (10,23) node[anchor=center] {$\pi_1$};
\draw (27,23) node[anchor=center] {$\pi_2$};
\draw (43,23) node[anchor=center] {$\pi_3$};
\draw (70,23) node[anchor=center] {$\pi$};
\draw (9,-1.2) node[anchor=center] {\small $s(\pi_1)$};
\draw (57,-1.2) node[anchor=center] {\small $s(\pi)=s(\pi_1)$};
\draw (2,12) node[anchor=center] {\tiny $\nu(\pi_1)$};
\draw (46.02,20) node[anchor=center] {\tiny $\nu(\pi_3)$};
\draw (46,18.5) node[anchor=center] {\tiny $=$};
\draw (46.02,17.5) node[anchor=center] {\tiny $\nu(\pi_1)$};
\draw (47.3,16.3) node[anchor=center] {\tiny $-1$};
\draw (49.4,9) node[anchor=center] {\tiny $\nu(\pi)$};
\draw (50,7.8) node[anchor=center] {\tiny $=$};
\draw (49.1,7) node[anchor=center] {\tiny $\nu(\pi_1)$};
\end{tikzpicture}
\end{center}
\end{table}
 
Let $\P_{e}$ be the subset of $\P$ where the smallest part is necessarily a positive even integer. The above construction proves that the left hand side of \eqref{S_function_for_P_ped_EQN} is the generating function for the weighted count of partitions from $\P_e$ counted by the weight $+1$ or $-1$ depending on the number of parts in the partition being even or odd, respectively: \begin{equation}\label{P_e_gen_func}\sum_{\pi\in\P_e} (-1)^{\nu(\pi)}q^{|\pi|} = \sum_{n\geq 1} \frac{(-1)^n q^{2n}}{1-q^{2n}}\frac{(-q;q^2)_{n-1}}{(q^2,q^2)_{n-1}}q^{n-1}.\end{equation}The right-hand side of \eqref{S_function_for_P_ped_EQN}, by looking at the geometric series, can easily be interpreted combinatorially. This study proves \begin{equation}\label{unrefined_triangular_eqn}
\sum_{\substack{\pi\in\P_e,\\ |\pi|=N}} (-1)^{\nu(\pi)}  = (-1)^{N+1} \chi (N \not= \triangle \,).
\end{equation} where $N$ is a positive integer and $\triangle$ represents ``a triangular number." The simple observation \eqref{parity_equivalence} can be used on \eqref{unrefined_triangular_eqn} to simplify the equation.

\begin{theorem}Let $N$ be a positive integer. Then,\begin{equation}
\sum_{\substack{\pi\in\P_e,\\ |\pi|=N}} (-1)^{\nu_e(\pi)+1}  =  \chi (N \not= \triangle \,).
\end{equation} where $\triangle$ represents ``a triangular number."
\end{theorem}

Moreover, it is easy to see that the generating function for the weighted count of partitions from $\P$ counted by the weight $+1$ or $-1$ depending on the number of parts is clearly \begin{equation}\label{P_gen_func}
\sum_{\pi\in\P} (-1)^{\nu(\pi)}q^{|\pi|} = \frac{(q^2;q^2)_\infty}{(-q;q^2)_\infty} -1= \psi(-q)-1.
\end{equation} Hence, \eqref{S_function_for_P_ped_EQN}, \eqref{P_e_gen_func}, and \eqref{P_gen_func} together yields \begin{equation}\label{P_o_gen_func}\sum_{\pi\in\P_o} (-1)^{\nu(\pi)}q^{|\pi|} = \frac{1}{1+q}-1, \end{equation} where $\P_o$ is the subset of $\P$ where the smallest part is necessarily a positive odd integer. 

We note that the above study can be easily generalized by inserting an extra parameter $z$. The identities \eqref{S_function_for_P_ped_EQN} and \eqref{P_e_gen_func} turn into
\begin{equation}
\sum_{n\geq 1} \frac{(-1)^n q^{2n}}{1-q^{2n}}\frac{(-q/z;q^2)_{n-1}}{(q^2,q^2)_{n-1}}(zq)^{n-1} = \frac{(q^2;q^2)_\infty}{(-qz;q^2)_\infty} - \frac{\tiny 1}{1+zq}
\end{equation}
and \begin{equation}\label{Bes_Pak1}
\sum_{\pi\in\P_e} (-1)^{\nu(\pi)}z^{\nu_o(\pi)}q^{|\pi|} = \frac{(q^2;q^2)_\infty}{(-qz;q^2)_\infty} - \frac{\tiny 1}{1+zq},
\end{equation}
respectively. We also get the generalization of \eqref{P_gen_func}
\begin{equation}\label{Bes_Pak2}\sum_{\pi\in\P} (-1)^{\nu(\pi)}z^{\nu_o(\pi)}q^{|\pi|} = \frac{(q^2;q^2)_\infty}{(-qz;q^2)_\infty} -1.
\end{equation} 
Combining \eqref{Bes_Pak1} and \eqref{Bes_Pak2} and replacing $z$ by $-z$ we get the result \begin{equation}\label{Bes_Pak3}
\sum_{\pi\in\P_o} (-1)^{\nu_e(\pi)}z^{\nu_o(\pi)}q^{|\pi|} = \frac{1}{1-zq} -1,
\end{equation} which can also be found in \cite[Cor 4, p.1146]{BessenrodtPak}. The equation \eqref{Bes_Pak3} implies

\begin{theorem}
\begin{equation}
\sum_{\substack{\pi\in\P_o,\\ |\pi|=N,\\ \nu_o(\pi)=k}} (-1)^{\nu_e(\pi)} = \chi(N=k).
\end{equation}
\end{theorem}

We can step up our study on the set $\P$ by putting more restrictive conditions on the smallest part. Let $\P_{2,4}$ be the subset of $\P_e$ where the smallest part of a partition is necessarily $2$ mod $4$. Knowing the argument behind the generating function interpretation for $\P$, the generating function of $\P_{2,4}$ with the $\pm 1$ weight with respect to the number of parts can easily be written as \begin{equation}\label{P_2_gen_func}
\sum_{\pi\in\P_{2,4}} (-1)^{\nu(\pi)}q^{|\pi|} = \sum_{n\geq 1} \frac{(-1)^n q^{2n}}{1-q^{4n}}\frac{(-q;q^2)_{n-1}}{(q^2,q^2)_{n-1}}q^{n-1}.
\end{equation} 

We write the related analytic equality.

\begin{theorem}\label{P_2_analytic_theorem}
\begin{equation}\label{P_2_analytic_equation}
\sum_{n\geq 1} \frac{(-1)^n q^{2n}}{1-q^{4n}}\frac{(-q;q^2)_{n-1}}{(q^2;q^2)_{n-1}}q^{n-1} = \frac{1}{1-q}\sum_{n\geq 0} (-1)^n q^{n^2} - \frac{1}{1-q^2}.
\end{equation}
\end{theorem}  

\begin{proof}
By multiplying both sides of \eqref{P_2_analytic_equation} with $2(1+q)$ and adding 1 to both sides, we see that one can apply the $q$-Gauss sum \eqref{q_Gauss} where $(a,b,c,q,z) = (-1,-1/q,-q^2,q^2,-q^3)$ to the left-hand side. Showing the equality of the right-hand side to the outcome product of the $q$-Gauss sum is a simple task of combining like terms and using the Gauss identity \eqref{Theta_phi_prod}.
\end{proof}

The right-hand side of \eqref{P_2_analytic_equation} can be studied further to get exact formulas.
\begin{align}
\nonumber \frac{1}{1-q}\sum_{n\geq 0} (-1)^n q^{n^2} &= \sum_{k\geq 0} q^k \sum_{n\geq 0} (-1)^n q^{n^2}\\
\nonumber&= 1+q^4+q^5+q^6+q^7+q^8+q^{16}+q^{17}+q^{18}+q^{19}\dots\\
&= 1+ \sum_{N\geq 1}\sum_{j\geq 1}\chi((2j)^2\leq N < (2j+1)^2\,) q^N,\label{phi_over_1-q_formula}
\end{align}
where $\chi$ is as defined in \eqref{truth_function}. Also from the geometric series\begin{equation}
\label{Geo_evens}\frac{1}{1-q^2} = 1+q^2+q^4+q^6+q^8+q^{10}+\dots .\end{equation}
Therefore, combining \eqref{phi_over_1-q_formula} and \eqref{Geo_evens}, we get\begin{align}
\nonumber \frac{1}{1-q}\sum_{n\geq 0} (-1)^n q^{n^2}-\frac{1}{1-q^2}&= -q^2+q^5+q^7-q^{10}-q^{12}-q^{14}+q^{17}+q^{19}+\dots\\
\nonumber&= \sum_{N\geq 1}\sum_{j\geq 1}(\chi(N\text{ is odd})\chi((2j)^2< N < (2j+1)^2)\\
\label{truth functions_of_P_2}&\hspace{2cm}-\chi(N\text{ is even})\chi((2j-1)^2< N < (2j)^2)\,) q^N.
\end{align}

Combining \eqref{P_2_gen_func}, \eqref{P_2_analytic_equation}, and \eqref{truth functions_of_P_2} we get the interesting explicit formula for the weighted count of partitions from the set $\P_{2,4}$.
\begin{theorem}
\begin{align*}
\sum_{\substack{\pi\in\P_{2,4},\\ |\pi|=N}} (-1)^{\nu(\pi)}\hspace{-.1cm} &= \hspace{-.1cm}\sum_{j\geq 1}(\chi(N\text{ is odd})\chi((2j)^2< N < (2j+1)^2)-\chi(N\text{ is even})\chi((2j-1)^2< N < (2j)^2))\\
&=\left\lbrace \begin{array}{rl} 1, &\text{if $N$ is odd and in between an even square and the following odd square,}\\ -1, &\text{if $N$ is even and in between an odd square and the following even square,}\\
0, &\text{otherwise.} \end{array}\right.
\end{align*}
\end{theorem}

Let $\P_{3,4}$, similar to $P_{2,4}$, be the subset of $\P_o$ where the smallest part of a partition is necessarily $3$ mod $4$. Adding a single 1 to the smallest part of a partition from $\P_{2,4}$ is a bijective map from the set $\P_{2,4}$ to $\P_{3,4}$. Therefore, writing the analogous generating function of weighted count of partitions from $\P_{3,4}$ is rather easy and only requires multiplying \eqref{P_2_gen_func} with and extra $q$. This proves the following theorem.
\begin{theorem}
\begin{align*}
\sum_{\substack{\pi\in\P_{3,4},\\ |\pi|=N}} (-1)^{\nu(\pi)}\hspace{-.1cm} &= \hspace{-.1cm}\sum_{j\geq 1}(\chi(N\text{ is even})\chi((2j)^2< N < (2j+1)^2)-\chi(N\text{ is odd})\chi((2j-1)^2< N < (2j)^2)\,)\\
&=\left\lbrace \begin{array}{rl} 1, & \text{if $N$ is even and in between an even square and the following odd square,}\\ -1, & \text{if $N$ is odd and in between an odd square and the following even square},\\
0, & \text{otherwise.} \end{array}\right.
\end{align*}
\end{theorem}
The combination of the weighted generating functions accounts for every number that is not a perfect square. This interesting relation can be represented as follows.

\begin{theorem}
\begin{equation}\label{Unrefined_difference}
\sum_{\substack{\pi\in\P_{3,4},\\ |\pi|=N}} (-1)^{\nu(\pi)} -\sum_{\substack{\pi\in\P_{2,4},\\ |\pi|=N}} (-1)^{\nu(\pi)} = (-1)^N \chi(N\not=\square).
\end{equation}
\end{theorem}

This result, in a sense, is complementary to Alladi's identity, Theorem~\ref{Alladi_squares_THM}. Also, the right-hand side formula also appears in the recent study of Andrews and Yee \cite[Thm 3.2, p.10]{Andrews_Yee_Top_Bottom_Heavy} as the same weighted count with respect to the number of parts of bottom-heavy partitions (a specific subset of overpartitions). The interested reader is invited to examine the relation between the set of bottom-heavy partitions, $\P_{2,4}$, and $\P_{3,4}$.

Once again one can simplify the argument of \eqref{Unrefined_difference} with the observation \eqref{parity_equivalence}.

\begin{theorem}
\begin{equation}\label{Refined_difference}
\sum_{\substack{\pi\in\P_{3,4},\\ |\pi|=N}} (-1)^{\nu_e(\pi)} -\sum_{\substack{\pi\in\P_{2,4},\\ |\pi|=N}} (-1)^{\nu_e(\pi)} = \chi(N\not=\square).
\end{equation}
\end{theorem}

\section{Overpartitions with no parts divisible by 3}\label{Section_not_0_mod_3}

In this section we do the weighted interpretation of an identity of Ramanujan \cite[E. 4.2.8, p. 85]{LostNotebook_2}. We write this identity in an equivalent form for the ease of interpretation purposes.

\begin{theorem}[Ramanujan]\label{Ramanujan_THM}
\begin{equation}\label{Ramanujan_EQN}
 \frac{(-q;q^3)_\infty(-q^2;q^3)_\infty}{(q;q^3)_\infty(q^2;q^3)_\infty} -1=\sum_{n\geq 1} \frac{(-q;q)_{n-1}}{(q;q)_{n-1}}\frac{2q^n}{1-q^n}\frac{q^{n^2-n}}{(q;q^2)_n}.
\end{equation}
\end{theorem}

Identity \eqref{Ramanujan_EQN} also appears in the Slater's list \cite[6, p. 152]{Slater} with a misplaced exponent type typo.

It is clear that the left-hand side of \eqref{Ramanujan_EQN} is the generating function for the number of overpartitions where no part is $0$ mod $3$. Let $\C$ be the set of all non-empty partitions with no parts divisible by 3. We focus our interest in the combinatorial interpretation of the right-hand side of \eqref{Ramanujan_EQN}. Let $n$ be a fixed positive integer. The factors \begin{equation}
 \frac{(-q;q)_{n-1}}{(q;q)_{n-1}}\frac{2q^n}{1-q^n}
\end{equation} of the right-hand side of \eqref{Ramanujan_EQN} is the generating function for the number of overpartitions, $\bar{\pi}_1$, into parts $\leq n$ where the part $n$ appears at least once. When counting the total number of overpartitions of this type of partitions, we can instead count the partitions $\pi_1$ into parts $\leq n$ where the part $n$ appears at least once with the weight $2^{\nu_d(\pi_1)}$ as in \eqref{over_connect_regular_abstract}. The remaining factor \[\frac{q^{n^2-n}}{(q;q^2)_n}\] can be split into two in the interpretation. The term $q^{n^2-n}$ is the generating function of the partitions of type $\pi_2=(2,4,6,\dots, 2(n-1))$ in frequency notation as $n^2-n$ is double a triangle number. The term $(q;q^2)_n^{-1}$ is the generating function for the number of partitions, $\pi_3$, into odd parts $\leq 2n-1$. 

It is clear that among the parts of $\pi_1$, $\pi_2$, and $\pi_3$ the largest possible part-size is $2n-1$. Even if $2n-1$ is not a part of $\pi_3$, the second largest possible part $2n-2$ is a part of $\pi_2$. Therefore, given $\pi_1$, $\pi_2$, and $\pi_3$ we can directly find the respective $n$. We merge (add the parts' frequencies of) these three partitions into a new partition and look at the number of possible sources for different part sizes. The partition $\pi$ has one appearance of all the even parts $\leq 2n-2$ coming from $\pi_2$, any extra appearance of an even number (which is necessarily $\leq n$) must be coming from the partition $\pi_1$ and should be counted with the overpartition weights. The odd parts $\leq n$ can either be coming from the partition $\pi_1$ or $\pi_3$. These parts need to be counted with both the overpartition weights and normally to account for both possibilities. All the other parts' source partitions can uniquely be identified so they would be counted with trivial weight 1. 

Let $\R$ be the set of partitions, where 
\begin{enumerate}[i.] 
\item all parts $\leq 2n-1$ for some integer $n>0$,
\item all even integers $\leq 2n-2$ appears as parts,
\item $n$ appears with the frequency $f_n \geq 1+ \chi(n\text{ is even})$,  
\item no even part $>n$ repeats. 
\end{enumerate} Clearly, \[n := n(\pi)= \frac{\text{largest even part of }\pi}{2} + 1.\] Define the statistics 
\begin{align}
\label{delta}\delta(\pi) &= \sum_{j= 1}^{n-1} \chi(f_{2j}> 1),\\
\nonumber\gamma(\pi) &= (\chi(n\text{ is even})+2\cdot f_n\cdot\chi(n\text{ is odd}))\prod_{2j+1<n} (2f_{2j+1}+1),\end{align}
and
\begin{equation}
\mu (\pi) = 2^{\delta(\pi)} \cdot \gamma(\pi),
\end{equation} for $\pi\in \R$. We have the following identity.

\begin{theorem}\label{Last_theorem1}
\begin{equation}
\sum_{\pi\in\C} 2^{\nu_d(\pi)}q^{|\pi|} = \sum_{\pi\in\R} \mu(\pi) q^{|\pi|}.
\end{equation}
\end{theorem}

One example of Theorem~\ref{Last_theorem1} will be given in Table~\ref{Table_last_theorem}.

In Ramanujan's entry \cite[E. 4.2.9, p. 86]{LostNotebook_2}, \begin{equation}\label{Ramanujan_EQN2} \frac{(-q;q^3)_\infty(-q^2;q^3)_\infty}{(q;q^3)_\infty(q^2;q^3)_\infty} =\sum_{n\geq 0} \frac{q^{n^2}(-q;q)_{n}}{(q;q)_{n}(q;q^2)_{n+1}}\end{equation} we see the same product of \eqref{Ramanujan_EQN}. The sum on the right-hand side of \eqref{Ramanujan_EQN2} can also be interpreted as a weighted partition count for a special subset of partitions. This is rather analogous to $\R$. Let the set $\Q$ be the set of partitions $\pi$, where
\begin{enumerate}[i.]
\item the largest part is $=2{n}-1$ for some integer $n> 0$,
\item all odd integers $\leq 2n-1$ appear as a part,
\item and no even parts $>n$ appear. 
\end{enumerate} Clearly here
\[n := \frac{\text{largest part of }\pi+1}{2}.\]
A similar weight to $\mu$ can be defined on $\Q$ as follows
\begin{equation}
\eta(\pi) := 2^{\nu_{d,e}(\pi)} (\chi(n\text{ is even})\cdot (1+\chi(f_{n}=0)\,)+2\cdot f_{n}\cdot\chi(n\text{ is odd})\,) \prod_{2j+1<n} (2f_{2j+1}-1)
\end{equation} where $\nu_{d,e}(\pi)$ is the number of different even parts of $\pi$. Hence, we have the identity

\begin{theorem}\label{Last_theorem}
\begin{equation}
\sum_{\pi\in\C} 2^{\nu_d(\pi)}q^{|\pi|} = \sum_{\pi\in\R} \mu(\pi) q^{|\pi|} = \sum_{\pi\in\Q} \eta(\pi) q^{|\pi|}.
\end{equation}
\end{theorem}

The example of this result is included in Table~\ref{Table_last_theorem}. From that table, it appears that there exists a weight, norm, and $n$-value preserving bijection from $\R$ to $\Q$. We would like to leave the discovery of this bijection for a motivated reader.

\begin{table}[htb]\caption{Example of Theorem~\ref{Last_theorem} with $|\pi|=7$.}\label{Table_last_theorem}
\begin{center}\vspace{-.5cm}
\[\begin{array}{ccc||ccc||ccc}
&\pi\in\C  & 2^{\nu_d(\pi)} &  \pi\in\R		&	n	&  \mu(\pi) & \pi\in\Q		&	n	&  \eta(\pi)\\[-2ex]& & && \\
&(1,2,4) 	& 2^3		&   (1^7) 			& 	1	&	14		& (1^7)			&	1	&	14\\
&(2,5)		& 2^2		&   (1^3,2^2)		& 	2	&	14		& (1^4,3)		&	2	&	14\\
&(1^2,5)	& 2^2		&   (1,2^3) 		&	2	&	6		& (1^2,2,3)		&	2	&	6\\
&(1^3,4) 	& 2^2		&   (2^2,3)	 		&	2	&	2		&(1,3^2)		&	2	&	2\\
&(1^5,2) 	& 2^2		&     				&		&			&&&\\
&(1^3,2^2)	&2^2		&					& 		&			&&&\\
&(1,2^3) 	&2^2 		&    				& 		&			&&&\\
&(7)		&2			&					&		&		&&&\\
&(1^7)		&2			&					&		&			&&&\\\hline
\text{Total:}&		&  36 & & &36&&&36\\
\end{array}\]
\end{center}
\end{table}

\section{Acknowledgement}
Authors would like to thank George Andrews for his kind interest, and Jeramiah Hocutt for his careful reading of the manuscript.

\end{document}